\documentclass[12pt]{article}
\usepackage{amsthm}
\usepackage{tikz}
\usepackage{fullpage}

\newtheorem{thm}{Theorem}

\newtheorem{con}[thm]{Conjecture}
\newtheorem{lem}{Lemma}

\begin{document}
\title{Nordhaus-Gaddum and other bounds for the sum of squares of the positive eigenvalues of a graph}
\author{Clive Elphick\thanks{\texttt{clive.elphick@gmail.com}}\quad\quad and Mustapha Aouchiche\thanks{\texttt{mustapha.aouchiche@gerad.ca}}}
\maketitle
\begin{abstract}

Terpai \cite{terpai11} proved the Nordhaus-Gaddum bound that $\mu(G) + \mu(\overline{G}) \le 4n/3 - 1$, where $\mu(G)$ is the spectral radius of a graph $G$ with $n$ vertices. Let $s^+$ denote the sum of the squares of the positive eigenvalues of $G$. We prove that $\sqrt{s^{+}(G)} + \sqrt{s^+(\overline{G})} < \sqrt{2}n$ and conjecture that $\sqrt{s^{+}(G)} + \sqrt{s^+(\overline{G})} \le 4n/3 - 1.$ We have used AutoGraphiX and Wolfram Mathematica to search for a counter-example. We also consider Nordhaus-Gaddum bounds for $s^+$ and bounds for the Randi\'c index.

\end{abstract}

\section{Introduction}

Let $G$ be  a graph with $n$ vertices, $m$ edges and chromatic number $\chi(G)$. Let $\overline{G}$ denote the complement of $G$ and $A$ denote the adjacency matrix of $G$. Let $\mu(G) = \mu_1(G) \ge ... \ge \mu_n(G)$ denote the eigenvalues of $A$. The inertia of $A$ is the ordered triple $(\pi, \nu, \gamma)$ where $\pi$, $\nu$ and $\gamma$ are the numbers (counting multiplicities) of positive, negative and zero eigenvalues of $A$ respectively. Let

\[
s^+ = \sum_{i=1}^\pi \mu_i^2 \qquad \mbox{  and  } \qquad s^- = \sum_{i=n-\nu+1}^n \mu_i^2.
\]

Note that 
$$
\sum_{i=1}^n \mu_i^2 = s^+ + s^- = tr(A^2) = 2m,
$$
where $tr(A^2)$ denotes the trace of the matrix $A^2$.

\medskip
Wocjan and Elphick \cite{wocjan13} introduced the spectral parameters $s^+$ and $s^-$ and conjectured that $1 + s^+/s^- \le \chi(G)$. They also demonstrated that this bound is often better than the classic inequality due to Hoffman: $1 + \mu/|\mu_n| \le \chi(G)$. Ando and Lin \cite{ando15} then used linear algebra to prove that:

\begin{equation}\label{eq:ando}
1 + \max{\left(\frac{s^+}{s^-} , \frac{s^-}{s^+}\right)} \le \chi(G).
\end{equation}

Elphick \emph{et al} \cite{elphick15} conjectured that for connected graphs:

\begin{equation}\label{eq:clive}
\min{(s^+ , s^-)} \ge n - 1 \mbox{  or equivalently  } \max{(s^+ , s^-)} \le 2m - n + 1,
\end{equation}
and proved this result for almost all graphs, including regular and bipartite graphs. They also proved that $s^- \le n^2/4$. 

\medskip

Stanley \cite{stanley87} proved the following upper bound on the spectral radius of a graph:

\[
\mu_1 \le \frac{\sqrt{8m + 1} - 1}{2}.
\]

Wu and Elphick \cite{wu15} strengthened this bound by proving that:

\[
\sqrt{s^+} \le \frac{\sqrt{8m + 1} - 1}{2}.
\]

AutoGraphiX (AGX) is a software program \cite{vns14,vns1,vns5} devoted to conjecture--making in graph theory. Conjectures 3 and 5 in this paper were not generated by AGX but have been checked for counter-examples by AGX.

\section{Nordhaus-Gaddum bounds}

In 1956, Nordhaus and Gaddum \cite{norgad} proved that
$$ 
2\sqrt{n} \le \chi(G) + \chi(\bar{G}) \le n + 1 
\qquad \mbox{ and } \qquad n \le \chi(G) \cdot \chi(\bar{G}) \le \frac{(n + 1)^{2}}{4}.
$$
Finck \cite{finck} showed that these bounds are sharp (taking floors and ceilings if necessary) and characterized extremal graphs. Similar bounds, now known as Nordhaus--Gaddum type inequalities, were obtained for a large number of graph invariants by a variety of authors. For more details about Nordhaus--Gaddum type inequalities, see the survey \cite{norgadsurvey} as well as the references therein. 

\subsection{Bounds for $\sqrt{s^+}$}

Amin and Hakimi \cite{amin1972} and Nosal \cite{nosal70} independently proved that for every graph $G$:

\[
n - 1 \le \mu(G) + \mu(\overline{G}) < \sqrt{2}(n - 1).
\]

The lower bound is exact for regular graphs. 

Adopting an approach first used by Nikiforov \cite{nikiforov02}, the above result can be extended as follows.

\begin{thm}
For any graph $G$ of order $n$,

\[
n - 1 \le \sqrt{s^{+}(G)} + \sqrt{s^+(\overline{G})} < \sqrt{2}n.
\]

\end{thm}

\begin{proof}

The lower bound is obvious and exact for complete graphs. 

From (\ref{eq:ando}) we have that $s^+(G) \le 2m(\chi(G) - 1)/\chi(G).$ Therefore:

\[
\sqrt{s^{+}(G)} + \sqrt{s^+(\overline{G})} \le \sqrt{2m\frac{(\chi(G) - 1)}{\chi(G)}} + \sqrt{\left(n(n - 1) - 2m\right)\frac{(\chi(\overline{G}) - 1)}{\chi(\overline{G})}}.
\]

So using Cauchy-Schwarz:

\begin{equation}\label{eq:vlado}
\sqrt{s^{+}(G)} + \sqrt{s^+(\overline{G})} \le \sqrt{\left(2 - \frac{1}{\chi(G)} - \frac{1}{\chi(\overline{G})}\right) n(n - 1)} < \sqrt{2}n.
\end{equation}

\end{proof}

Similarly, at least one of $G$ and $\overline{G}$ is connected. So if we assume $G$ is disconnected and that inequality (\ref{eq:clive}) is true then:

\[
\sqrt{s^{+}(G)} + \sqrt{s^+(\overline{G})} \le \sqrt{2m} + \sqrt{n(n - 1)-2m -n + 1}.
\]

So using Cauchy-Schwarz, we strengthen Nosal's bound, since:

\[
\sqrt{s^{+}(G)} + \sqrt{s^+(\overline{G})} \le \sqrt{2(n(n - 1) - (n - 1)} = \sqrt{2}(n - 1).
\]

Following Nikiforov \cite{nikiforov07}, let:
\[
f_1 = \max_{v(G)=n}{(\mu(G) + \mu(\overline{G}))}.
\]

Nikiforov \cite{nikiforov07} proved the following lower bound and conjectured the following upper bound: 

\begin{equation}\label{eq:vlado2}
\frac{4n}{3}  - 2 \le f_1 \le \frac{4n}{3} + O(1).
\end{equation}

The authors of \cite{aouchiche08,openprob} used AGX to devise the following conjecture. Before the statement of the conjecture, recall that a {\it complete split graph} $CS_{n,\alpha}$ on $n$ vertices with independence number $\alpha$ is the complement of the graph composed of a clique $K_\alpha$ and $n-\alpha$ isolated vertices.

\begin{con}

For any simple graph $G$

\[
f_1 = \frac{4n}{3} - \frac{5}{3} + f(n)
\]

where $f(n) = 0$ if $n \equiv 2$ (mod 3); 

$f(n) = \left(\sqrt{(3n - 2)^2 +8} - (3n - 2)\right)/6$ if $n \equiv 1$ (mod 3); and

$f(n) = \left(\sqrt{(3n - 1)^2 + 8} - (3n  - 1)\right)/6$ if $n\equiv 0$ (mod 3).

This bound is exact if and only if $G$ or $\overline{G}$ is a complete split graph with an independent set on $\lfloor \frac{n}{3} \rfloor$ vertices (and also on $\lceil \frac{n}{3} \rceil$ vertices if $n \equiv 2$ (mod 3)).

\end{con}

Terpai \cite{terpai11} used mathematical analysis and graphons to very nearly solve this conjecture by proving that:

\[
\frac{4n}{3} - 2 \le f_1 \le \frac{4n}{3} -1.
\]

The proof of the lower bound in (\ref{eq:vlado2}) uses $G = K_r + \overline{K_{n-r}}$, and it is notable that these graphs and their complements, which are complete split graphs,  both have only one positive eigenvalue. This suggests the following conjecture.

\begin{con}\label{sqrts+conj}
For any graph $G$:
\[
\max_{v(G)=n}{(\mu(G) + \mu(\overline{G}))} = \max_{v(G)=n}{\left(\sqrt{s^{+}(G)} + \sqrt{s^+(\overline{G})}\right)} = \frac{4n}{3} - \frac{5}{3} + f(n).
\]
\end{con}

We have tested this conjecture using AGX and the 10,000s of named graphs with up to 40 vertices in the Wolfram Mathematica database, and found no counter-example.

\bigskip

\subsection{Bounds for $s^+$}

Similarly we can consider upper and lower bounds for $s^+(G) + s^+(\overline{G})$. 

\begin{thm}

For any graph $G$ (provided inequality (\ref{eq:clive}) is true):

\[
\frac{(n - 1)^2}{2} < s^+(G) + s^+(\overline{G}) \le (n - 1)^2.
\]

\end{thm}

\begin{proof}~\\

\emph{Upper bound}

Since $s^+(G) + s^-(G) = 2m$:

\[
s^+(G) + s^-(G) + s^+(\overline{G}) + s^-(\overline{G}) = 2m + (n(n - 1) - 2m) = n(n - 1).
\]

It is well known that if $G$ is disconnected then $\overline{G}$ is connected. Therefore assuming (\ref{eq:clive}) is correct, then $s^-(G) + s^-(\overline{G}) \ge n - 1.$ Hence:

\[
s^+(G) + s^+(\overline{G}) = n(n - 1) - (s^-(G) + s^-(\overline{G})) \le n(n - 1) -(n - 1) = (n - 1)^2.
\]

This bound is exact when $G = K_n$ or the empty graph.

\medskip
\emph{Lower bound}

Using that $\mu(G) \ge 2m/n$ we have that:

\[
s^+(G) + s^+(\overline{G}) \ge \frac{4m^2}{n^2} + \frac{(n(n - 1) - 2m)^2}{n^2}.
\]

This sum is minimised when $m(G) = m(\overline{G}) = n(n-1)/4$, so

\[
s^+(G) + s^+(\overline{G}) \ge \frac{(n - 1)^2}{2}.
\]

For this lower bound to be exact would require a \emph{regular graph}  $G$ such that $G$ and $\overline{G}$ both have precisly one positive eigenvalue and $m(G) = m(\overline{G})$. Smith \cite{smith70} proved that complete multipartite graphs are the only connected graphs with only one positive eigenvalue. Consequently no such graph can exist, and $s^+(G) + s^+(\overline{G}) > (n -1)^2/2$. 

\end{proof}

The best we can state about the lower bound is the next conjecture, for which we need the following definitions. A {\it strongly regular graph} (SRG),  with parameters $(n, k, \lambda, \nu)$ is a $k$-regular graph on $n$ vertices in which every two adjacent vertices have $\lambda$ common neighbours and every two non-adjacent vertices have $\nu$ common neighbours. A {\it conference graph} is a SRG with parameter set $(4t+1, 2t, t-1,t)$ where $t$ is a positive integer and $n = 4t + 1$ is a sum of squares of two integers.

\begin{con}
For any graph $G$,
\[
s^+(G) + s^+(\overline{G}) \ge \frac{(n - 1)^2}{2} + \frac{(n - 1)(n + 1 -2\sqrt{n})}{4} = \frac{(n - 1)(3n - 1 - 2\sqrt{n})}{4}.
\]
with equality if and only if $G$ is a conference graph.
\end{con}

Regarding this conjecture, we can make the following observations. If $G = SRG(n, k, \lambda, \nu)$ then $\overline{G} = SRG(n, n-k-1, n-2-2k+\nu, n - 2k + \lambda).$  Consequently if $G$ is a conference graph then $\overline{G}$ is a conference graph with the same parameter set. The spectrum of a conference graph is:

\[
(n - 1)/2 , ((\sqrt{n} - 1)/2)^{(n-1)/2} , (-(\sqrt{n} + 1)/2)^{(n-1)/2}.
\]

Therefore:

\[
s^+(G) = s^+(\overline{G}) = \left(\frac{n - 1}{2}\right)^2 + \frac{1}{2}(n - 1)\left(\frac{\sqrt{n} - 1}{2}\right)^2 = \frac{(n - 1)^2}{4} + \frac{1}{8}(n - 1)(n + 1 - 2\sqrt{n}).
\]

We have tested this conjecture using AGX and named graphs in Wolfram Mathematica and found no counter-example. 

\bigskip
It is worth noting that if $E(G)$ denotes graph energy, that is $E(G) = \sum_{i=1}^n |\mu_i(G)|$, then Nikiforov and Yuan \cite{nikiforov13} proved that:
\[
E(G) + E(\overline{G}) \le (n - 1)(1 + \sqrt{n})
\]
with equality if and only if $G$ is a conference graph.

There are some similarities between this result and Conjecture 5, since both involve positive eigenvalues and the same extremal graphs. It may be possible  to adapt their proof, which uses Weyl's inequalities for Nordhaus-Gaddum sums of symmetric non-negative matrices with zero diagonal and with all entries $\le 1$, to prove Conjecture 5.

\section{Bounds involving Randi\'c index}

Let $R(G)$ denote the Randi\'c index of graph $G$, which has applications in mathematical chemistry \cite{randic1975}. It is defined as follows, where $E$ denotes the edge set of $G$:

\[
R(G) = \sum_{ij  \in E} \frac{1}{\sqrt{d_i.d_j}}.
\]

To prove the next theorem, we need the following lemma proved in \cite{favaron93}.
\begin{lem}[\cite{favaron93}]\label{fav-lem}
For any connected graph $G$,
$$
\frac{m}{\mu} \le R(G)
$$
with equality if and only if $G$ is regular or bipartite semiregular.
\end{lem}

Favaron \emph{et al} \cite{favaron93} proved that:

\[
|\mu_n| \le R(G).
\]

This bound can be strengthened as follows.

\begin{thm}
For any connected graph $G$
\[
\sqrt{s^-} \le R(G)
\]
with equality if and only if $G$ is a complete bipartite graph.
\end{thm}

\begin{proof}

Using successively the AM-GM inequality, the Cauchy-Schwarz inequality and that $\sum \mu_i^2 = 2m$, we have that:

\begin{equation}
2\sqrt{\mu\sqrt{s^-}} \le \mu + \sqrt{s^-} \le \sqrt{2(\mu^2 + s^-)} \le 2\sqrt{m}. \label{ineq}	
\end{equation}

Consequently $\sqrt{s^-} \le m/\mu$,  and from Lemma~\ref{fav-lem} $m/\mu \le R(G)$, which completes the proof.

For the characterization of the extremal graphs, it is easy to see that $\sqrt{s^-} = R(G)$ for any complete bipartite graph.

Now assume that $G$ is such that $\sqrt{s^-} = R(G)$. From $\sqrt{s^-} \le m/\mu \le R(G)$  necessarily $m/\mu = R(G)$ which implies equality in (\ref{ineq}), {\it i.e.}, $\mu \sqrt{s^-} = m$, and then $\mu = \sqrt{s^-} = \sqrt{m}$. Thus, $s^- = m$ and therefore $\mu^2 = m = s^+$, which means that $G$ contains exactly one positive eigenvalue, {\it i.e.} $G$ is necessarily a complete multipartite graph. In addition, according to Lemma~\ref{fav-lem}, $G$ is $\mu$-regular or bipartite semiregular with degrees $d$ and $d'$ such that $\mu = d\cdot d'$. So in the case of a bipartite semiregular graph, it is complete. 

If $G$ is regular, then $m = \mu^2 = 4m^2/n^2$ and then $G$ is a complete multipartite $\frac{n}{2}$-regular graph, which is necessarily $K_{\frac{n}{2},\frac{n}{2}}$. 
\end{proof}

Experiments with AGX suggest the following conjectures.

\begin{con}
For any connected graph $G$
\[
\frac{2\sqrt{n-1}}{n-3 + 2\sqrt{2}} \le \frac{\sqrt{s^+}}{R(G)} \le \frac{2(n-1)}{n}.
\]
The lower (resp. upper) bound is reached if and only if $G$ is the path $P_n$ (resp. the complete graph $K_n$).
\end{con}

We can prove this upper bound, using Lemma 1 and provided inequality (2) is correct, as follows.

\[
R(G) \ge \frac{m}{\mu} \ge \frac{m}{\sqrt{s^+}} = \frac{m\sqrt{s^+}}{s^+} \ge \frac{m\sqrt{s^+}}{2m - n + 1} \ge \frac{n\sqrt{s^+}}{2(n - 1)}.
\]

Clearly $2m = n(n - 1)$ only for $K_n$.

\begin{con}
For any triangle-free graph $G$
\[
R(G) \ge \sqrt{s^+} 
\]
with equality if and only if $G$ is a complete bipartite graph.
\end{con}

We can prove this conjecture for bipartite graphs, using (1), as follows:

\[
\mu^2 \le s^+ \le \frac{2m(\chi(G) - 1)}{\chi(G)} = m.
\]

Therefore using Lemma 1

\[
\sqrt{s^+} \le \sqrt{m } \le \frac{m}{\mu} \le R(G).
\]

\section{Conclusion}

Csikv\'ari \cite{csikvari09} used the Kelmans transformation to prove that $ \mu(G) + \mu(\overline{G}) \le (1 + \sqrt{3})n/2 - 1$ and Terpai \cite{terpai11} used a graphon equivalent of the Kelmans transformation to prove his bound. We do not know if there exists an equivalent of the Kelmans transformation for $\sqrt{s^+}$.

\medskip

It should be noted, that if $G'$ is a subgraph of $G$ then $\mu(G') \le \mu(G)$. However Godsil \cite{godsil11} has searched all graphs on 9 vertices and found 5 for which $s^+(G') > s^+(G)$. This is indicative of the difficulty of proving results using $s^+$.

\end{document}